\documentclass{article}
\pdfoutput=1
\usepackage[utf8]{inputenc}
\usepackage{amsmath}
\usepackage{amssymb}
\usepackage{amsthm}
\usepackage{mathtools}
\usepackage[colorlinks=true,linkcolor=blue, urlcolor=blue, linktocpage=true]{hyperref}
\usepackage{siunitx}
\usepackage{setspace}
\usepackage[pass]{geometry}
\usepackage{authblk}
\makeatletter
\def\UrlAlphabet{%
	\do\a\do\b\do\c\do\d\do\e\do\f\do\g\do\h\do\i\do\j%
	\do\k\do\l\do\m\do\n\do\o\do\p\do\q\do\r\do\s\do\t%
	\do\u\do\v\do\w\do\x\do\y\do\z\do\A\do\B\do\C\do\D%
	\do\E\do\F\do\G\do\H\do\I\do\J\do\K\do\L\do\M\do\N%
	\do\O\do\P\do\Q\do\R\do\S\do\T\do\U\do\V\do\W\do\X%
	\do\Y\do\Z}
\def\UrlDigits{\do\1\do\2\do\3\do\4\do\5\do\6\do\7\do\8\do\9\do\0}
\g@addto@macro{\UrlBreaks}{\UrlOrds}
\g@addto@macro{\UrlBreaks}{\UrlAlphabet}
\g@addto@macro{\UrlBreaks}{\UrlDigits}
\makeatother

\numberwithin{equation}{section}

\newtheorem{theorem}{Theorem}[section]
\newtheorem{definition}[theorem]{Definition}
\newtheorem{corollary}[theorem]{Corollary}
\newtheorem{proposition}[theorem]{Proposition}
\newtheorem{lemma}[theorem]{Lemma}
\newtheorem{remark}[theorem]{Remark}
\newtheorem{example}[theorem]{Example}

\makeatletter
\newcommand{\address}[1]{\gdef\@address{#1}}
\newcommand{\email}[1]{\gdef\@email{\url{#1}}}
\newcommand{\@endstuff}{\par\vspace{\baselineskip}\noindent\small
\begin{tabular}{@{}l}\scshape\@address\\\textit{E-mail address:} \@email\end{tabular}}
\AtEndDocument{\@endstuff}
\makeatother
\title{Jordan Constant of $\mathrm{PGL}_3(K)$}
\author{Yijue Hu}
\address{Department of Mathematics, National Research University Higher School of Economics}
\email{yijue.hu@gmail.com}
\usepackage{lipsum}
\date{}

\begin{document}

\maketitle

%\newgeometry{a4paper,left=1.5in,right=1.5in}

\begin{abstract}
We determine the Jordan constants of groups $\mathrm{GL}_2(K)$, $\mathrm{SL}_2(K)$, $\mathrm{PGL}_2(K)$ and $\mathrm{PGL}_3(K)$ for any given field $K$ of characteristic 0.
\end{abstract}

\section{Introduction}

%We will classify all the finite subgroups of $\mathrm{PGL}_3(K)$, and calculate the Jordan constants of subgroups of each type. 

One of the commonly used approaches to study an infinite group is to study its finite subgroups. However, sometimes we want to avoid classifying all of its finite subgroups thoroughly because it may be too complicated. To do this, we will find the following definition useful, which originally comes from~\mbox{\cite[Definition 2.1]{popov}}. It measures to which extent a group is ``similar" to an abelian group.

\begin{definition}\label{def1.1}
For an arbitrary group $\Gamma$, its Jordan constant is defined as the minimal integer $J(\Gamma)$ such that any finite subgroup of $\Gamma$ has an abelian normal subgroup with index not exceeding $J(\Gamma)$.
\end{definition}

Throughout this paper, we work over a field $K$ of characteristic 0, unless stated otherwise. We use $\omega$ to denote a third primitive root of unity.

Our main result deals with the Jordan constant of $\mathrm{PGL}_3(K)$.

\begin{theorem}\label{thm1.4}
The Jordan constant of $\mathrm{PGL}_3(K)$ can attain only the following values: 360, 168, 60, 24, 12, 6. More explicitly:

(i) $J(\mathrm{PGL}_3(K))=360$ if and only if $\omega\in K$ and $\sqrt{5}\in K$;

(ii) $J(\mathrm{PGL}_3(K))=168$ if and only if at least one of $\omega$ and $\sqrt{5}$ is not in $K$, and $\sqrt{-7}\in K$;

(iii) $J(\mathrm{PGL}_3(K))=60$ if and only if $\sqrt{5}\in K$, but neither $\omega$ nor $\sqrt{-7}$ is in~$K$;

(iv) $J(\mathrm{PGL}_3(K))=24$ if and only if $\sqrt{5}\not\in K$, $\sqrt{-7}\not\in K$ and 

\hspace{20pt} (a) either $\omega\in K$, 

\hspace{20pt} (b) or $-1$ is the sum of two squares in $K$, and there exists $\zeta\in K$ such 

\hspace{37pt} that $2\zeta^2$ is a root of unity;

(v) $J(\mathrm{PGL}_3(K))=12$ if and only if $-1$ is the sum of two squares in $K$, none of $\omega, \sqrt{5}$, and $\sqrt{-7}$ lies in $K$, and for any element $\zeta\in K$, $2\zeta^2$ is not a root of unity;

(vi) $J(\mathrm{PGL}_3(K))=6$ if and only if $-1$ is not the sum of two squares in $K$, and none of $\omega, \sqrt{5}$, and $\sqrt{-7}$ lies in $K$.
\end{theorem}

\begin{example}\label{eg1.5}
We can consider the most popular fields:
$$J(\mathrm{PGL}_3(\mathbb{C}))=360,\  J(\mathrm{PGL}_3(\mathbb{R}))=60,\  J(\mathrm{PGL}_3(\mathbb{Q}))=6.$$
\end{example}

% Note finite subgroups of $\mathrm{SL}_3(\bar{K})$ can be classified thoroughly. The classification, together with some basic definitions, will be introduced in section 2. By embedding any finite subgroup of $\mathrm{PGL}_3(K)$ into $\mathrm{PGL}_3(\bar{K})$ and lifting it back to $\mathrm{SL}_3(\bar{K})$ via the natural projection, we will give criteria for existence of primitive subgroups (refer to Definition 2.1) in section 3. Finally we will prove the main result: Theorem 2.1, in section 4.

To prove Theorem 1.2, we find the necessary and sufficient conditions for the group $\mathrm{PGL}_3(K)$ to contain a given finite primitive subgroup (see Definition \ref{def4.1}). This may be considered as a partial generalization of \cite[Proposition 1.1]{Beauville}, where the necessary and sufficient conditions for the existence of a given finite subgroup in $\mathrm{PGL}_2(K)$ are found. We expect that our results can be used to determine the possible Jordan constant of the group of birational automorphisms of $\mathbb{P}^2$ over any field $K$ of characteristic 0, which is currently known only over $\mathbb{C}, \mathbb{R}$ and $\mathbb{Q}$, see \cite{Yasinsky}. 

As an auxiliary step of the proof of the main result, Jordan constant of~\mbox{$\mathrm{GL}_2(K)$} is determined in Section 3. We are also able to determine Jordan constants of $\mathrm{SL}_2(K)$ and $\mathrm{PGL}_2(K)$ as by-products of the proof of~\mbox{Theorem \ref{thm1.3}}. 

\begin{theorem}\label{thm1.3}
The Jordan constant of $\mathrm{GL}_2(K)$ can attain only the following values: 60, 24, 12, 2. More explicitly:

(i) $J(\mathrm{GL}_2(K))=60$ if and only if $-1$ is the sum of two squares in $K$ and~\mbox{$\sqrt{5}\in K$};

(ii) $J(\mathrm{GL}_2(K))=24$ if and only if $-1$ is the sum of two squares in~$K$,~\mbox{$\sqrt{5}\not\in K$} and there exists $\zeta\in K$ such that $2\zeta^2$ is a root of unity;

(iii) $J(\mathrm{GL}_2(K))=12$ if and only if $-1$ is the sum of two squares in~$K$,~\mbox{$\sqrt{5}\not\in K$} and for any element $\zeta\in K$, $2\zeta^2$ is not a root of unity;

(iv) $J(\mathrm{GL}_2(K))=2$ if and only if $-1$ is not the sum of two squares in $K$.
\end{theorem}

\begin{corollary}\label{3.10}
The Jordan constant of $\mathrm{SL}_2(K)$ can attain only the following values: 60, 24, 12, 2. More explicitly:

(i) $J(\mathrm{GL}_2(K))=60$ if and only if $-1$ is the sum of two squares in $K$ and~\mbox{$\sqrt{5}\in K$};

(ii) $J(\mathrm{GL}_2(K))=24$ if and only if $-1$ is the sum of two squares in~$K$,~\mbox{$\sqrt{5}\not\in K$} and either $\sqrt{2}\in K$ or $\sqrt{-2}\in K$;

(iii) $J(\mathrm{GL}_2(K))=12$ if and only if $-1$ is the sum of two squares in $K$, and none of $\sqrt{5}, \sqrt{2}$ and $\sqrt{-2}$ is in $K$;

(iv) $J(\mathrm{GL}_2(K))=2$ if and only if $-1$ is not the sum of two squares in $K$.
\end{corollary}

\begin{proposition}\label{thm1.2}
The Jordan constant of $\mathrm{PGL}_2(K)$ can attain only the following values: 60, 6, 2. More explicitly:

(i) $J(\mathrm{PGL}_2(K))=60$ if and only if $-1$ is the sum of two squares in $K$ and~\mbox{$\sqrt{5}\in K$};

(ii) $J(\mathrm{PGL}_2(K))=6$ if and only if $-1$ is the sum of two squares in $K$ and~\mbox{$\sqrt{5}\not\in K$};

(iii) $J(\mathrm{PGL}_2(K))=2$ if and only if $-1$ is not the sum of two squares in $K$.
\end{proposition}

\begin{remark}
In some cases the criterion that $-1$ is the sum of two squares in $K$ holds automatically. For example, in claim (ii) of Theorem \ref{thm1.3}, if $2\zeta^2$ is a root of unity whose order $n$ is divisible by $4$, then $-1=x^2+x^2$, where~\mbox{$x=(2\zeta^2)^{\frac{n}{4}}$}. Moreover, in claim (ii) of Corollary \ref{3.10}, one has $-1=(\sqrt{-2})^2+1^2$ when~\mbox{$\sqrt{-2}\in K$}.
\end{remark}

The plan of the paper is as follows. We compute the Jordan constants of several finite groups in Section 2, based on which the Jordan constant of $\mathrm{GL}_2(K)$ is determined in Section 3. In Section 4, we classify all types of finite subgroups in $\mathrm{PGL}_3(K)$, and the criterion of the existence of each primitive subgroup is given in Section 5. Finally, the main result is proved in Section 6.

\noindent \textbf{Notation}: Throughout this essay, the cyclic group of order $r$ is denoted by~$C_r$; the dihedral group of order $2r$ is denoted by $D_r$; the finite field of order $q$ is denoted by $\mathbb{F}_q$; the algebraic closure of a field $K$ is denoted by $\bar{K}$; and the identity matrix is denoted by $I$.

\section{Preliminaries}
In this section, we compute the Jordan constants of several groups which will appear as subgroups of $\mathrm{PGL}_2(K)$ or $\mathrm{GL}_2(K)$.

\begin{lemma}\label{lem2.1}
(i) $J(S_n)=n!$, $J(A_n)=\frac{n!}{2}$, for $n\geq 5$.

(ii) $J(S_4)=6$, $J(A_4)=3$.

(iii) $J(D_n)=2$ for $n>2$.
\end{lemma}
\begin{proof}
For claim (i), it's well-known that $A_n$ is a simple group when $n\geq 5$, so~\mbox{$J(A_n)=|A_n|$}. Also, $A_n$ is a normal subgroup of $S_n$. If $S_n$ contains some other proper normal subgroup $N$, then $N\cap A_n$ is a normal subgroup of $A_n$, which implies that the intersection is either $\{1\}$ or $A_n$. The latter case implies~\mbox{$N=A_n$}. The previous case tells us $N\cap [S_n,S_n]=\{1\}$, hence $N$ is included in the center of  $S_n$, which is trivial. That is, $S_n$ has no non-trivial normal subgroups other than $A_n$, while $A_n$ is non-abelian.

For claim (ii), the subgroup $$N=\{(1), (12)(34), (13)(24), (14)(23)\}$$ is abelian and normal in $A_4$. Note that every group of order 6 is either $C_6$ or symmetric group $S_3$. Clearly $A_4$ does not contain $C_6$. Also $A_4$ does not contain~$S_3$, as $S_3$ has exactly 3 order-$2$ elements, none of which commutes with the others. So $J(A_4)=3$.

Now let's consider normal subgroups of $S_4$. First of all, $N$ is also abelian and normal in $S_4$. Furthermore, a subgroup is normal only if it's a union of conjugacy classes and contains $\{1\}$. Conjugacy classes of $S_n$ are determined by cycle type, so it's easy to observe that the conjugacy classes in $S_4$ have cardinality 1, 3, 6 or 8. Therefore $S_4$ cannot have normal subgroup of order 6 or 8, and the only normal subgroup of order 12 is $A_4$, which is non-abelian.

For the last claim, note that $C_n\mathrel{\unlhd} D_n$, and $D_n$ is non-abelian when $n>2$. 
% When $n=2$, $D_2\cong C_2\times C_2$ is an abelian group.
\end{proof}

\begin{lemma}\label{lem2.2}
Let $G$ be a central extension of $D_n$ by a finite cyclic subgroup $C_r$, where $n\geq 3$, then $J(G)=2$.
\end{lemma}
\begin{proof}
We have a short exact sequence 
$$1\rightarrow C_r\mathop{\rightarrow}^i G\mathop{\rightarrow}^{\pi} D_n\rightarrow 1.$$
Let $C_n\mathrel{\unlhd} D_n$ be the abelian normal subgroup of $D_n$, then $\pi^{-1}(C_n)$ is abelian and normal in $G$, with index 2. Note that $G$ is not abelian as $D_n$ is not abelian when $n\geq 3$, hence $J(G)=2$.
\end{proof}

% \begin{lemma}
% Let $G$ be a nontrivial central extension of $A_4$ by a finite cyclic subgroup $C_r$, then $J(G)\leq 12$. The equality holds when $r=2$.
% \end{lemma}
% \begin{proof}
% Obviously $C_r\mathrel{\unlhd} G$, so $J(G)\leq |A_4|=12$.

% When $r=2$, we know the Schur multiplier of $A_4$ is $H_2(A_4,\mathbb{Z})=C_2$, so the second group cohomology
% \begin{align*}
% H^2(A_4,C_2)&\cong \mathrm{Hom}(H_2(A_4,\mathbb{Z}),C_2)\oplus \mathrm{Ext}^1_{\mathbb{Z}}(A_4^{ab},C_2)\\
% &=\mathrm{Hom} (C_2,C_2)\oplus \mathrm{Ext}^1_{\mathbb{Z}}(C_3,C_2)\cong C_2,
% \end{align*}
% where $G^{ab}$ means the abelianization of a group $G$.

% That is, $A_4$ has the unique non-trivial central extension by $C_2$, and $\mathrm{SL}_2(\mathbb{F}_3)$ is turned out to be such a group. It only has $C_2$ as its normal abelian subgroup, therefore $J(\mathrm{SL}_2(\mathbb{F}_3))=12$.
% \end{proof}

\begin{lemma}\label{lem2.4}
Let $G$ be a nontrivial central extension of $A_5$ by a finite cyclic subgroup $C_r$, then $J(G)=60$.
\end{lemma}
\begin{proof}
Note that a maximal normal abelian subgroup of $G$ always contains~$C_r$, and it can be projected to a normal abelian subgroup of $A_5$, which is trivial. Therefore the maximal normal abelian subgroup of $G$ is exactly $ C_r$, and~\mbox{$J(G)=|A_5|=60$}.
\end{proof}

\section{Jordan constants of $\mathrm{PGL}_2(K)$ and $\mathrm{GL}_2(K)$}
We determine the Jordan constant of $\mathrm{GL}_2(K)$ in this section. The result is needed when we calculate the Jordan constants of finite intransitive subgroups of $\mathrm{PGL}_3(K)$. Proofs of Proposition \ref{thm1.2} and Corollary \ref{3.10} are also given in this section.
% We also obtain the Jordan constants of $\mathrm{PGL}_2(K)$ and $\mathrm{SL}_2(K)$ as by-products in Theorem \ref{thm1.2} and Corollary \ref{3.10}, respectively.

% In $\mathrm{SL}_3(K)$, there is an obvious class of finite subgroups 
% $$G=\{\begin{pmatrix}
% g & \\
%  & \det(g)^{-1} 
% \end{pmatrix}|g\in G'\}$$
% where $G'$ is a finite subgroup of $\mathrm{GL}_2(K)$. 

% After applying the canonical projection $\pi:\mathrm{SL}_3(K)\rightarrow \mathrm{PGL}_3(K)$, we obtain a particular class of finite subgroups in  $\mathrm{PGL}_3(K)$:
% $$G=\{\begin{pmatrix}
% \det(g)g & \\
%  & 1 
% \end{pmatrix}|g\in G'\}$$
% where $G'$ is a finite subgroup of $\mathrm{GL}_2(K)$. 

% We denote $\Tilde{G}=\{\det(g)g | g\in G\}$ and study $J(\Tilde{G})$ in this section.

% Let $H$ be the collection of all scalar matrix in $\Tilde{G}$, therefore $H$ is a finite acyclic subgroup. Consider the short exact sequence
% $$1\rightarrow H\rightarrow \Tilde{G}\rightarrow \Tilde{G}/H \rightarrow 1$$

% Note $\Tilde{G}/H$ is a finite subgroup of $\mathrm{PGL}_2(K)$, and we have the following theorem.

\begin{theorem}\label{3.1}
For an algebraically closed field $K$ of characteristic 0, a finite subgroup $G$ of $\mathrm{PGL}_2(K)$ must be of one of the following 5 types: $C_r$, $D_r$, alternating groups $A_4, A_5$ and symmetric group $S_4$.
\end{theorem}
\begin{proof}
We refer to \cite[Chapter X, Section 101-103]{Mi-Bl-Di} for determination of all finite subgroups of $\mathrm{SL_2}(K)$, and then consider the canonical projection to $\mathrm{PGL}_2(K)$.
\end{proof}

\begin{theorem}\label{3.2} 
(i) $\mathrm{PGL}_2(K)$ contains $C_r$ and $D_r$ if and only if $K$ contains~\mbox{$\alpha+ \alpha^{-1}$} for some primitive $r$-th root of unity $\alpha$.

(ii) $\mathrm{PGL}_2(K)$ contains $A_4$ and $S_4$ if and only if $-1$ is the sum of two squares in $K$.

(iii) $\mathrm{PGL}_2(K)$ contains $A_5$ if and only if $-1$ is the sum of two squares and~\mbox{$\sqrt{5}\in K$}.
\end{theorem}
\begin{proof}
See \cite[Proposition 1.1]{Beauville}.
\end{proof}

Now we prove Proposition \ref{thm1.2}.
\begin{proof}
There are five types of subgroups $\mathrm{PGL}_2(K)$, listed in Theorem \ref{3.1}. Note that $A_4$ and $S_4$ appear in $\mathrm{PGL}_2(K)$ simultaneously, according to Theorem \ref{3.2}, and by Lemma \ref{lem2.1}, one has $J(S_4)>J(A_4)$. Hence we only consider $S_4$. 

From Lemma \ref{lem2.1}, we know that the largest possible Jordan constant is~$J(A_5)=60$, and the criterion for the existence of $A_5$ in $\mathrm{PGL}_2(K)$ is given in Theorem \ref{3.2}, thus (i) is proved. 

If $A_5$ does not exist in $\mathrm{PGL}_2(K)$, we will then consider $S_4$, thus (ii) can be proved again by Lemma \ref{lem2.1} and Theorem \ref{3.2}. 

If $S_4$ does not exist in $\mathrm{PGL}_2(K)$ either, then finite subgroups of $\mathrm{PGL}_2(K)$ are all of the form $C_r$ or $D_r$, whose Jordan constants are at most $2$. It's always true that $\omega^{-1}+\omega=-1\in K$, hence $D_3$ is a subgroup of $\mathrm{PGL}_2(K)$ by~\mbox{Theorem \ref{3.2}}. Therefore $J(\mathrm{PGL}_2(K))=J(D_3)=2$, and (iii) is proved.
\end{proof}

\begin{proposition}\label{3.3}
Let $a,b\in K$ such that $-1=a^2+b^2$, consider the following matrices in $\mathrm{PGL}_2(K)$:
$$A=\begin{pmatrix}
-a & b\\
b &a
\end{pmatrix},\ 
B=\frac{1}{2}\begin{pmatrix}
-1+a+b & -1+a-b\\
1+a-b & -1-a-b
\end{pmatrix},\ 
C=\begin{pmatrix}
-a+1 & b\\
b &a+1
\end{pmatrix}.$$
Then $A$ and $B$ generate $A_4$, while $B$ and $C$ generate $S_4$. Moreover, any subgroup of $\mathrm{PGL}_2(K)$ which is isomorphic to $A_4$ is conjugate to $\left\langle A,B\right\rangle$, and any subgroup of $\mathrm{PGL}_2(K)$ which is isomorphic to $S_4$ is conjugate to $\left\langle B,C\right\rangle$.
\end{proposition}
\begin{proof}
It's known that $$A_4\cong\left\langle x,y\mid x^3=y^2=(xy)^3=1\right\rangle,$$ according to \cite[p138, Beispiel 19.8]{Huppert}. Here $x=(123), y=(12)(34)$.

We take $x=B, y=A$, then all relations are satisfied. Therefore $A$ and $B$ will generate a group which is a quotient of $A_4$, which is either $A_4$ or~\mbox{$A_4/C_2\times C_2\cong C_3$}. The latter case is impossible, because $y=A$ is an order-2 element, which does not exist in $C_3$. As a result, $\langle A, B\rangle\cong A_4$.

Also we have $$S_4\cong\left\langle x,y,z\mid x^2=y^2=z^2=(xy)^3=(yz)^3=(zx)^2=1\right\rangle,$$ according to \cite[p138, Beispiel 19.7]{Huppert}. Here $x=(14), y=(24), z=(23)$.

We take $x=CB, y=BC$, and $z=C^2B^2C$, then all relations are satisfied. Therefore they will generate a group which is a quotient of $S_4$, which is one of $S_4$,~\mbox{$S_4/A_4\cong C_2$} or $S_4/C_2\times C_2\cong S_3$. Note that $C=zyx, B=yC^{-1}$ and~\mbox{$A=C^2$}, so the group contains $A$ and $B$, hence the subgroup generated by $A$ and $B$, which is $A_4$, thus our group must be $S_4$. Additionally, $\langle x,y,z\rangle=\langle B, C\rangle$, so $B$ and $C$ generate $S_4$.

Finally, from \cite[Theorem 4.2]{Beauville}, we know that $\mathrm{PGL}_2(K)$ contains only one conjugacy class of subgroups isomorphic to $A_4$, and one conjugacy class of subgroups isomorphic to $S_4$.
\end{proof}

Now we study finite subgroups of $\mathrm{GL}_2(K)$. Let $\pi:\mathrm{GL}_2(K)\rightarrow \mathrm{PGL}_2(K)$ be the canonical projection.

\begin{corollary}\label{3.5}
Let $G$ be a finite subgroup of $\mathrm{GL}_2(K)$, then $G$ is a central extension of $C_r$, $D_r$, $A_4$, $S_4$ or $A_5$ by a finite cyclic subgroup.
% $$H\cong\{\begin{pmatrix}
% \alpha & 0\\
% 0 & \alpha 
% \end{pmatrix} |\ \alpha \text{\ is\ a\ root\ of\ unity}\}.$$
\end{corollary}
\begin{proof}
Let $H\subset G$ be the subgroup of $G$ consisting of all scalar matrices, then it's finite and cyclic. After applying $\pi$ to $G$, the claim will be an obvious corollary of Theorem \ref{3.1}. 
\end{proof}

\begin{lemma}\label{3.6}
Let $n\geq 2$ be an integer, and $r$ is an integer coprime to $n$. If there is a cyclic subgroup of order $r$ in $\mathrm{PGL}_n(K)$, then there is a cyclic subgroup of the same order in $\mathrm{SL}_n(K)$.
\end{lemma}
\begin{proof}
We can choose integers $u, v$ such that $un+vr=1$ since~\mbox{$\gcd (r,n)=1$}. Assume $A$ is a matrix in $\mathrm{GL}_n(K)$ such that $A^r=aI, a\in K^{\times}$, then we have~\mbox{$\det(A)^r=a^n$}. Consider $$B= \frac{1}{\det(A)^ua^v}A,$$ then $B^r=I$.
% Suppose that $r=nk+1$ for some $k$. Assume $A$ is a matrix in $\mathrm{GL}_n(K)$ such that $A^r=aI, a\in K^{\times}$, then we have $$\det(A) = \left(\frac{a}{\det(A)^k}\right)^n.$$ Consider $$B= \frac{\det(A)^k}{a}A,$$ then $B^r=I$.

% Now suppose that $n=3$ and $r=3k+2$. We have $$\det(A) =\frac{\det(A)^{3k+3}}{a^3}$$in this case. And we may consider $$B= \frac{a}{\det(A)^{k+1}}A.$$
\end{proof}

\begin{lemma}\label{5.5}
Let $n\geq 2$ be an integer, and $r$ is an integer coprime to $n$. Then for any finite subgroup $G\subset\mathrm{PGL}_n(K)$ which is generated by its elements of order $r$, there exists a finite subgroup $\Tilde{G}\subset \mathrm{SL}_n(K)$ such that $\pi(\Tilde{G})=G$. If~\mbox{$|\Tilde{G}|=n|G|$}, then $K$ contains a primitive $n$-th root of unity.
\end{lemma}
\begin{proof}
Pick a set of generators of $G$, consisting of elements of order $r$. According to Lemma \ref{3.6}, each element in the set can be lifted to $\mathrm{SL}_n(K)$. Then let~\mbox{$\Tilde{G}\subset\mathrm{SL}_n(K)$} be the subgroup generated by those liftings.
\end{proof}

\begin{proposition}\label{3.4}
Let $G$ be a finite subgroup of $\mathrm{GL}_2(K)$.

(i) If $\pi(G)\cong A_4$, then $J(G)=12$.  

(ii) If $\pi(G)\cong S_4$, then $J(G)=24$. 
\end{proposition}
\begin{proof}
Let's first assume that $\pi(G)=\left\langle A,B\right\rangle$ or $\left\langle B, C\right\rangle$, where $A, B$ and $C$ are given in Proposition \ref{3.3}. Let $N\cong C_2\times C_2$ be the unique nontrivial abelian normal subgroup of both $A_4$ and $S_4$. Then $g_1=A$, and $$g_2=\begin{pmatrix}
0 & 1\\
-1 &0
\end{pmatrix}=BAB^{-1}$$
give a pair of generators of $N$.

For any lifting $h_1\in G$ of $g_1$, and  $h'\in G$ of $B$, $h_2=h'h_1h'^{-1}$ will be a lifting of $g_2$, and it does not commute with $h_1$.

From Corollary \ref{3.5} we know that $G$ is a central extension of $A_4$ or $S_4$ by a finite cyclic subgroup $C_r$, which consists of all the scalar matrices in $G$. Note that a maximal abelian normal subgroup of $G$ must contain $C_r$. Let $H\subset G$ be such a subgroup. Then $\pi(H)$ is an abelian normal subgroup in $A_4$ or $S_4$, so $\pi(H)=\{1\}$ or $N$. If $\pi(H)=N$, then all liftings of $g_1$ and $g_2$ lie in $H$. In particular, we have $h_1\in H$ and $h_2\in H$. However, these two elements do not commute. The contradiction implies that $\pi(H)=\{1\}$. Hence $H=C_r$, and $$J(G)=\frac{|G|}{r}=|\pi(G)|.$$

Now we consider any $G$ such that $\pi(G)\cong A_4$ or $S_4$. By Proposition \ref{3.3} we know that $\pi(G)$ is conjugate to $\left\langle A,B\right\rangle$ or $\left\langle B, C\right\rangle$. Then our $G$ is conjugate to the subgroup we discussed in above paragraphs, by the same conjugation. Note that conjugation does not change Jordan constant.
\end{proof}

\begin{proposition}\label{3.7}
We have a finite subgroup $G\subset \mathrm{GL}_2(K)$ such that $\pi(G)\cong A_4$ if and only if $-1$ is the sum of two squares in $K$. If this is the case, then~\mbox{$J(G)= 12$}.
\end{proposition}
\begin{proof}
If $\pi(G)\cong A_4$, then $\mathrm{PGL}_2(K)$ contains $A_4$. Then $-1$ is the sum of two squares in $K$, by claim (ii) in Theorem \ref{3.2}.

Conversely, if $-1$ is the sum of two squares in $K$, then $\mathrm{PGL}_2(K)$ contains~$A_4$, again according to Theorem \ref{3.2}. Note that $A_4$ is generated by its order-3 elements, so we can lift all of them to $\mathrm{SL}_2(K)$ by Lemma \ref{5.5}. Let $G$ be the group generated by those lifted elements. 

Finally, we obtain $J(G)=12$ according to Proposition \ref{3.4}. 
\end{proof}

\begin{proposition}\label{3.8}
We have a finite subgroup $G\subset \mathrm{GL}_2(K)$ such that $\pi(G)\cong S_4$ if and only if $-1$ is the sum of two squares in $K$,and there exists $\zeta\in K$ such that $2\zeta^2\in K$ is a root of unity. If this is the case, then $J(G)=24$.
\end{proposition}
\begin{proof}
Suppose such $G$ exists, then $\mathrm{PGL}_2(K)$ contains $S_4$, therefore $-1$ is the sum of two squares in $K$, by claim (ii) in Theorem \ref{3.2}. Moreover, from Proposition \ref{3.3}, we know that $\mathrm{PGL}_2(K)$ contains only one conjugacy class of subgroups isomorphic to $S_4$. So we may assume it is generated by the matrices $B$ and~$C$ given in Proposition \ref{3.3}, without loss of generality. In particular, there is a matrix $g\in G$ such that $\pi(g)=C$. Assume $g=\zeta C$, then $\det(g)=2\zeta^2$ is a root of unity, as $|G|<\infty$.

Conversely, if $-1$ is the sum of two squares in $K$, we can firstly construct~$S_4$ in~$\mathrm{PGL}_3(K)$ via Proposition \ref{3.3}. Note that $B$ itself is an order-3 matrix in~$\mathrm{GL}_2(K)$. If $2\zeta^2$ is an $l$-th primitive root of unity in $K$, let $g=\zeta C$, and $G$ be the subgroup generated by $B$ and $g$. We observe that $\pi(G)\cong S_4$, and~\mbox{$g^{8l}=I$}. Let $$\mathrm{SL}^{24l}(K)= \{h\in\mathrm{GL}_2(K)\mid \ \det(h)^{24l}=1\}.$$
Let $\pi_{24l}$ be the restriction of $\pi$ to $\mathrm{SL}^{24l}(K)$, then $G\subset \pi_{24l}^{-1}(S_4)$, therefore $G$ is finite.

Finally, we obtain $J(G)=24$ according to Proposition \ref{3.4}.
\end{proof}

\begin{proposition}\label{3.9}
We have a finite subgroup $G\subset \mathrm{GL}_2(K)$ such that~\mbox{$\pi(G)\cong A_5$} if and only if $-1$ is the sum of two squares in $K$ and $\sqrt{5}\in K$. If this is the case, then $J(G)=60$.
\end{proposition}
\begin{proof}
If $\pi(G)\cong A_5$, then $\mathrm{PGL}_2(K)$ contains $A_5$. Then $-1$ is the sum of two squares in $K$, and $\sqrt{5}\in K$, by claim (iii) in Theorem \ref{3.2}. 

Conversely, if $-1$ is the sum of two squares in $K$ and $\sqrt{5}\in K$, then $\mathrm{PGL}_2(K)$ contains $A_5$, again according to Theorem \ref{3.2}. Note that $A_5$ is generated by its order-3 elements, so we can lift all of them to $\mathrm{SL}_2(K)$ by Lemma \ref{5.5}. Let $G$ be the group generated by those lifted elements. 

Finally, we obtain $J(G)=60$ according to Lemma \ref{lem2.4}. 
\end{proof}

Now we are ready to prove Theorem \ref{thm1.3}.
\begin{proof}
If a finite subgroup $G\subset \mathrm{GL}_2(K)$ is a central extension of $C_r$ by a finite cyclic subgroup, then $G$ is still abelian, hence $J(G)=1$.

If $G$ is a central extension of $D_r$ by a finite cyclic subgroup, where $r>2$, then $J(G)=2$, by Lemma \ref{lem2.2}. Moreover, such subgroup always exists: $D_3\cong S_3$ has an obvious 2-dimensional irreducible faithful representation.

Let $G$ be a finite subgroup of $\mathrm{GL}_2(K)$, with $\pi(G)=A_4, S_4$ or $A_5$, then~$J(G)$ equals to $12, 24$ or $60$, respectively. We start from the case which give the largest possible Jordan constant. Then claim (i) is obtained by Proposition \ref{3.9}. If conditions in claim (i) are not satisfied, we may search for $\pi(G)=S_4$, then claim (ii) comes from Proposition \ref{3.8}. Furthermore, if conditions in both claim (i) and (ii) do not hold, we now search for $G$ such that $\pi(G)=A_4$, which is claim (iii), and it can be proved by referring to Proposition \ref{3.7}. Finally, if $\pi(G)$ cannot be any one of $A_4, S_4$ and $A_5$, then it should be $D_r$ or $C_r$. We already show that there always exists $G$ such that $\pi(G)=D_3$, with $J(G)=2$, so we come to claim (iv).
\end{proof}

Finally we provide a short proof of Corollary \ref{3.10}.
\begin{proof}
Note that in the proof of Theorem \ref{thm1.3} all the lifting matrices lie in~$\mathrm{SL}_2(K)$, except the lifting of the matrix $C$, which is one of the generators of~$S_4$, defined in Proposition \ref{3.3}. Assume that we have a lifting $\zeta C$ in $\mathrm{SL}_2(K)$, then $\zeta\in K$, and $1=\det (\zeta C)=2\zeta^2$. So $\zeta$ should be $\frac{1}{\sqrt{2}}$ or $\frac{1}{\sqrt{-2}}$, which is equivalent to say either $\sqrt{2}\in K$ or $\sqrt{-2}\in K$.
\end{proof}

% construction of central extension of $A_4$ by $C_2$:
% $$-1=a^2+b^2, a,b\in K$$
% https://math.stackexchange.com/questions/7144/generators-and-relations-for-a-4
% $$x=\begin{pmatrix}
% -a & b \\
% b & a
% \end{pmatrix}=(12)(34),
% y=\begin{pmatrix}
% -1+a+b & a-b-1 \\
% a-b+1 & -1-a-b
% \end{pmatrix}=(123).$$

% $xy=\begin{pmatrix}
% 1+a+b & 1+a-b\\
% -1+a-b & 1-a-b
% \end{pmatrix},$

\section{Finite subgroups in $\mathrm{PGL}_3(K)$}

In this section, we recall the concept of primitive subgroups, referring to \cite[Chapter 1, Definition 1.1]{Yau-Yu} and \cite[Chapter XI, Section 106]{Mi-Bl-Di}. We also list all types of finite subgroups in $\mathrm{PGL}_3(K)$.

\begin{definition}\label{def4.1}
We regard elements in $\mathrm{GL}_n(K)$ as linear automorphisms of a vector space  $V=K^n$. Consider a subgroup $G$ of $\mathrm{GL}_n(K)$.

1. We say $G$ is intransitive if we can decompose the vector space $V$ into a direct sum of more than one subspaces $V=\bigoplus\limits_{i}V_i$, such that $g(V_i)=V_i$, for all $i$ and all $g\in G$. If such a decomposition does not exist, then we say $G$ is transitive.

2. We say $G$ is imprimitive, if it's intransitive, and we can decompose the the vector space $V$ into a direct sum of more than one subspaces $V=\bigoplus\limits_{i}V_i$, such that for each $i$ and each $g\in G$, $g(V_i)\subset V_j$ for some $j$. 

3. We say $G$ is primitive, if it's neither transitive, nor imprimitive.
\end{definition}

\begin{theorem}\label{4.2}
For an algebraically closed field $K$ of characteristic 0, a finite subgroup $G$ of $\mathrm{SL}_3(K)$ is conjugate to a group of one of the following 12 types:

(A) diagonal group;

(B) group of the form $$\left\{\begin{pmatrix}
g & 0\\
0 & \det(g)^{-1} 
\end{pmatrix}\mid g\in \Tilde{G}\right\},$$ where $\Tilde{G}\cong G$ is a finite subgroup of $\mathrm{GL}_2(K)$;

(C) group generated by group of type (A) and $T$;

(D) group generated by group of type (C) and $R_{a,b,c}$ for some $a,b$ and $c$ in~$K$;

(E) group of order 108 generated by S, T and V;

(F) group of order 216 generated by (E) and $UVU^{-1}$;

(G) group of order 648 generated by (E) and $U$;

(H) simple group of order 60 isomorphic to the alternating group $A_5$;

(I) simple group of order 168 isomorphic to $\mathrm{PSL}_2(\mathbb{F}_7)$;

(J) group of order 180 generated by (H) and F;

(K) group of order 504 generated by (I) and F;

(L) group of order 1080 with quotient $G/F$ isomorphic to the alternating group $A_6$.

Here 
\begin{align*}
T=\begin{pmatrix}
0 & 1 & 0\\
0 & 0 & 1 \\
1 & 0 & 0
\end{pmatrix},\ 
S&=\begin{pmatrix}
1 & 0 & 0\\
0 & \omega & 0 \\
0 & 0 & \omega^2
\end{pmatrix}, \\
V=\frac{1}{\sqrt{-3}}\begin{pmatrix}
1 & 1 & 1\\
1 & \omega & \omega^2\\
1 & \omega^2 & \omega
\end{pmatrix},\ 
U&=\begin{pmatrix}
\epsilon & 0 & 0\\
0 & \epsilon & 0\\
0 & 0 & \epsilon\omega
\end{pmatrix},\ 
R_{a,b,c}=\begin{pmatrix}
a & 0 & 0\\
0 & 0 & b \\
0 & c & 0
\end{pmatrix}
\end{align*} with $abc=-1$.
Furthermore, $F=\{I,\omega I,\omega^2 I\}$ is the center of $\mathrm{SL}_3(K)$, and~\mbox{$\epsilon^3=\omega^2$}. 
%Elements $a, b, c, d, e, f$ lie in $K$, and $abc=-def=1$.
\end{theorem}
\begin{proof}
See \cite[Chapter XII]{Mi-Bl-Di} and \cite[Chapter 1]{Yau-Yu} for the case $K=\mathbb{C}$. For an arbitrary field $K$, note that every finite subgroup of $\mathrm{SL}_3(K)$ can be embedded into $\mathrm{SL}_3(\mathbb{C})$.
\end{proof}

\begin{remark}\label{4.3}
 Here groups of type (A) and (B) are intransitive; groups of type~(C) and (D) are imprimitive; and all the remaining groups are primitive.
\end{remark}

Consider the natural projection $\pi:\mathrm{SL}_3(\bar{K})\rightarrow \mathrm{PGL}_3(\bar{K})$. In the sequel, we adopt the following notation: we say a group is of some type in $\mathrm{PGL}_3(K)$, for example, of type (E), when the group is isomorphic to $\pi(G)$ for some finite subgroup $G\subset \mathrm{SL}_3(\bar{K})$ of type (E).

\begin{corollary}\label{4.4}
A finite subgroup $G$ of $\mathrm{PGL}_3(K)$ is isomorphic to a group of one of the following 10 types:

(A) abelian group;

(B) finite subgroup of $\mathrm{GL}_2(K)$;

(C) group generated by group of type (A) and $T$;

(D) group generated by group of type (C) and $R'_{a,b,c}$ for some $a,b$ and $c$ in~$K$;

(E) group of order 36, which is isomorphic to $(C_3\times C_3)\rtimes C_4$, generated by~$S$, $T$ and $V$;

(F) group of order 72 generated by (E) and $UVU^{-1}$;

(G) Hessian group of order 216, which is isomorphic to $(C_3\times C_3)\rtimes \mathrm{SL}_2(\mathbb{F}_3)$, generated by (E) and $U$;

(H) the simple alternating group $A_5$ of order 60;

(I) the simple group $\mathrm{PSL}_2(\mathbb{F}_7)$ of order 168;

(L) the simple alternating group $A_6$ of order 360.

Here 
$$
% T'_{a,b,c}=\begin{pmatrix}
% 0 & a & 0\\
% 0 & 0 & b \\
% c & 0 & 0
% \end{pmatrix},
R'_{a,b,c}=\begin{pmatrix}
a & 0 & 0\\
0 & 0 & b \\
0 & c & 0
\end{pmatrix} $$
with $abc\neq 0$.
\end{corollary}
\begin{proof}
Let $G$ be a finite subgroup of $\mathrm{PGL}_3(K)$, then $G$ can be regarded as a finite subgroup of $\mathrm{PGL}_3(\bar{K})$ via the natural embedding $\mathrm{PGL}_3(K)\hookrightarrow \mathrm{PGL}_3(\bar{K})$. Let $\Tilde{G}\subset\mathrm{SL}_3(\bar{K})$ be the preimage of $G$ via $\pi$. Then $\Tilde{G}$ is conjugate to one of the 12 types of groups in Theorem \ref{4.2}, and $G$ is isomorphic to the projection of such a group. 

Groups of type (A) may not be diagonal when the field $K$ is not algebraically closed, but they're still abelian. 

Now let $G$ be a finite non-abelian (if it's abelian, then sort it into type (A)) subgroup of $\mathrm{PGL}_3(K)$ such that the corresponding $\Tilde{G}\subset \mathrm{SL}_3(\bar{K})$ is conjugate to a group of type (B). That is, we have a unique decomposition of the vector field $\bar{K}^{3}=V\oplus U$, where $\dim V=2, \dim U=1$, such that $\pi(V)$ and $\pi(U)$ are the only invariant line and invariant point of $G$ acting on $\mathbb{P}^2_{\bar{K}}$. Pick an arbitrary element $\sigma\in \mathrm{Gal}(\bar{K}/K)$, note that $G=\pi(\Tilde{G})=\pi(\sigma(\Tilde{G}))$, since $G$ is defined over~$K$. And $\pi(\sigma(V))$ is also an invariant line of $G$, hence it must be $\pi(V)$. That is, the subspace $V$ is preserved by the action of all elements in the Galois group, hence it is defined over $K$. Similarly, $\pi(U)$ is the unique invariant point of $G$, so it's preserved by all elements in the Galois group, hence it's a point over $K$. Therefore, we find an invariant line and an invariant point outside the line for $G$ when it's regarded as a subgroup of $\mathrm{Aut}(\mathbb{P}^2_{K})$. That is, $G$ is conjugate to a group in which every matrix is of the form 
$$\begin{pmatrix}
g & 0\\
0 & a
\end{pmatrix},$$
where $g\in\mathrm{GL}_2(K), a\in K^{\times}$. Up to re-scaling, we may assume that $a=1$, and then $G$ is isomorphic to a finite subgroup in $\mathrm{GL}_2(K)$, consisting of all the $2\times 2$ matrices $g$ appearing in the upper left block. 

Note that for a group $\tilde{G}\subset \mathrm{SL}_3(\bar{K})$ of type (A), 
%no matter the kernel of the projection $\pi$ is contained in $\tilde{G}$ or not, 
the projection $\pi(\Tilde{G})$ will always be abelian, i.e. of type (A). Therefore for a group $\Tilde{G}\subset \mathrm{SL}_3(\bar{K})$ of type~(C) or (D), its projection $\pi(\Tilde{G})$ is again generated by a group of type (A) and some other specific matrices.

A group $\Tilde{G}\subset \mathrm{SL}_3(\bar{K})$ of type (E)--(G) or (J)--(L) contains the kernel of the projection $F$, therefore $G\cong \pi(\Tilde{G})$ is isomorphic to $\Tilde{G}/F$, and $|G|=|\Tilde{G}|/3$. On the other hand, groups $\Tilde{G}\subset \mathrm{SL}_3(\bar{K})$ of type (H) and (I) are simple, so~\mbox{$\pi(\Tilde{G})\cong G$}. In particular, projections of groups of type (H) and (J) to $\mathrm{PGL}_3(\bar{K})$ are isomorphic to each other, and the same happens with groups of type (I) and (K). 
\end{proof}

\begin{remark}\label{4.5}
We restrict ourselves to fields of characteristic 0 because Theorem \ref{4.2} does not hold for fields with positive characteristics. For instance, when~\mbox{$K= \bar{\mathbb{F}}_p$}, the group $\mathrm{PGL}_3(K)$ contains $\mathrm{PSL}_3(\mathbb{F}_q)$, where $q=p^{n}$. These subgroups are simple, hence the Jordan constant of $\mathrm{PGL}_3(K)$ reaches infinity. 
\end{remark}

% \begin{definition}
% 1. For a finite group $G$, the complexity of $G$ is defined as the minimal index of a normal abelian group in G.

% 2. For an arbitrary (possibly infinite) group $\Gamma$, we define its complexity as the maximal complexity 
% of all finite subgroups in $\Gamma$.

% The complexity of group $\Gamma$ is denoted as $J(\Gamma)$
% \end{definition}

\section{Primitive finite subgroups in $\mathrm{PGL}_3(K)$}

In this section, we give the criterion of the existence of each primitive subgroup in  $\mathrm{PGL}_3(K)$, for an arbitrary field $K$ of characteristic 0.

% \begin{lemma}
% If there is a cyclic subgroup of order n in $\mathrm{PGL}_3(K)$, where $n$ is not divisible by 3, then there is a cyclic subgroup of the same order in $\mathrm{SL}_3(K)$.
% \end{lemma}
% \begin{proof}
% If $n=3k+1$. Assume $A$ is a matrix in $\mathrm{GL}_3(K)$ such that $A^n=aI, a\in K^{\times}$. Then we have $\det(A) = (\frac{a}{\det(A)^k})^3$. Consider $B\coloneqq \frac{\det(A)^k}{a}A$, then $B^n=I$.

% If $n=3k+2$. We have $\det(A) =\frac{\det(A)^{3k+3}}{a^3}$in this case. And we may consider $B\coloneqq \frac{a}{\det(A)^{k+1}}A$.
% \end{proof}

\begin{proposition}\label{5.1}
  Let $n$ be an odd prime number, then $\mathrm{PGL}_3(K)$ has an cyclic subgroup of order $n$ if and only if there is an $n$-th primitive root of unity $\alpha$ and indices $1< i< j\leq n$, with $1+i+j\equiv 0 $ (mod n) such that $$\alpha+\alpha^i+\alpha^j\in K,\   \alpha^{-1}+\alpha^{-i}+\alpha^{-j}\in K.$$
\end{proposition}
\begin{proof}
The case $n=3$ is clear. For $n\neq 3$, we may lift a generator of the cyclic subgroup to a matrix $A\in \mathrm{SL}_3(K)$ of the same order, by Lemma \ref{3.6}. Let $f(x)$ be the characteristic polynomial of $A$. 
%Then at least one of the eigenvalues of $A$, say $\alpha$, is an $n$-th primitive root of unity.

Case 1: Suppose $f$ has a multiple root of order $3$, i.e. $f(x)=(x-\alpha)^3\in K[x]$, then $\alpha\in K$. Our conditions are clearly satisfied.

Case 2: Suppose $f$ has a multiple root of order $2$. Let $\zeta, \lambda$ be two distinct roots of $f$, and do the Euclidean Algorithm. Write $$\Phi_n(x)=f(x)q(x)+r(x),$$ with $\deg r(x)=1$ or $2$, where $\Phi_n(x)$ is the $n$-th cyclotomic polynomial. Note that both $\zeta$ and $\lambda$ are roots of $r(x)$, so $$r(x)=(x-\zeta)(x-\lambda)\in K[x],\ \zeta\lambda\in K.$$
If $\zeta\lambda=1$, then $\lambda=\zeta^{-1}$ and $\zeta+\lambda=\zeta+\zeta^{n-1}\in K$. We may pick $\alpha=\zeta$ and indices $(1,n-1,n)$ to fit the conditions. Otherwise, $K$ contains an $n$-th primitive root of unity $\zeta\lambda$, which will be chosen as $\alpha$.

Case 3: Suppose $f$ has three simple roots, denote one of them by $\alpha$, such that the other two can be presented as $\alpha^i$ and $\alpha^j$. The fact that $\alpha^{1+i+j}\in K$ implies that either $K$ contains an $n$-th primitive root of unity or $$1+i+j\equiv 0\ (\text{mod}\ n).$$ For the latter case, the coefficients of square and linear terms are $$-(\alpha+\alpha^i+\alpha^j), \  \alpha^{-1}+\alpha^{-i}+\alpha^{-j},$$ which are exactly what we want.

 Conversely, if we have such $\lambda=\alpha+\alpha^i+\alpha^j$ and $\eta=\alpha^{-1}+\alpha^{-i}+\alpha^{-j}$, consider
$$A=\left(
                                                                  \begin{array}{ccc}
                                                                    \lambda+1 & 1 & -\lambda-\eta-2 \\
                                                                    1 & 0 & 0 \\
                                                                    1 & 0 & -1 \\
                                                                  \end{array}
                                                                \right).$$
Its characteristic polynomial is $f(x)=x^3-\lambda x^2+\eta x-1$, so the three eigenvalues of $A$ are $\alpha,\alpha^i$ and $\alpha^j$, hence the order of $A$ is $n$.
\end{proof}

%\begin{proposition}
%$\mathrm{PGL}_3(K)$ contains subgroups of type (A), (C) and (D) if and only if the condition in Proposition 1 is satisfied for some prime number $n$.
%\end{proposition}

% This gives the criteria on the existence of finite subgroups of type (B) in $\mathrm{PGL}_3(K)$.

\begin{lemma}\label{5.2}
The group $\mathrm{PGL}_3(K)$ contains $C_3\times C_3$ if and only if $\omega\in K$.
\end{lemma}
\begin{proof}
If $\omega\in K$, then matrices $S$ and $T$ given in Theorem \ref{4.2} will generate~\mbox{$C_3\times C_3$}.

For the other direction, we will prove by contradiction. Assume $\omega$ is not in $K$, and $\mathrm{PGL}_3(K)$ contains $C_3\times C_3$.  Let $A, B\in \mathrm{GL}_3(K)$ denote  representatives of the generators of each $C_3$, respectively. Assume that~\mbox{$\det(A)=\lambda\in K$},~\mbox{$\det(B)=\mu\in K$}, then $A^3=\lambda I, B^3=\mu I$. 

If $\lambda\in K$ does not have a cubic root in $K$, we may construct the field extension $L=K[x]/(x^3-\lambda)$ of degree $3$. If $\omega\in L$, then $K\subset K(\omega)\subset L$ is a subextension of $L$ of degree 2, which is impossible, as $3$ is not divisible by 2. Therefore $L$ contains exactly one cubic root of $\lambda$ and does not contain $\omega$. We may repeat this procedure once again if $\mu$ does not have a cubic root in the new field, and finally obtain a larger field $K\subset L$ such that $L$ contains cubic root of both $\lambda$ and $\mu$, but not $\omega$. 

Now we may assume $A^3=B^3=I$, after replacing $A$ by $\frac{1}{\sqrt[\leftroot{-2}\uproot{2}3]{\lambda}}A$, $B$ by $\frac{1}{\sqrt[\leftroot{-2}\uproot{2}3]{\mu}}B$, and $K$ by $L$. These two matrices commute in $\mathrm{PGL}_3(K)$, thus there exists $\alpha\in K$ such that $AB=\alpha BA$. By computing determinants of matrices at both sides, we get $\alpha=1$ as $\omega\not\in K$. That is, they also commute in $\mathrm{SL}_3(K)$. Therefore $A$ and $B$ generate $C_3\times C_3$ in $\mathrm{SL}_3(K)$, thereby in $\mathrm{SL}_3(\bar{K})$, via the canonical embedding. Note that $C_3\times C_3$ has only one-dimensional irreducible representations since it's abelian, so we can diagonalize $A$ and $B$ simultaneously in $\mathrm{SL}_3(\bar{K})$. Eventually, we achieve a subgroup $C_3\times C_3\subset \mathrm{SL}_3(\bar{K})$ consisting of diagonal matrices, which does not contain any scalar matrices other than $I$. However, this is impossible, since $\mathrm{SL}_3(\bar{K})$ contains exactly 8 diagonal matrices of order 3, and two of them are scalar matrices.
%8 diagonal matrices of order 3 are required to form $C_3\times C_3$, while 
% Therefore $A$ and $B$ give a three dimensional representation of $C_3\times C_3$, which has only one-dimensional irreducible representation since it's abelian. So we can diagonalize $A$ and $B$ simultaneously in $\mathrm{SL}_3(\bar{K})$. It can be checked easily that the diagonalization of $A$ must be of the form $\text{diag}\{1,\omega,\omega^2\}$, and we have no other choice for $B$ to make the group generated by $A$ and $B$ an abelian group $C_3\times C_3\in \mathrm{PGL}_3(K)$.
\end{proof}

\begin{proposition}\label{5.3}
The group $\mathrm{PGL}_3(K)$ contains subgroups of type (E), (F) and~(G) if and only if $\omega\in K$.
\end{proposition}
\begin{proof}
If $\omega\in K$, then we refer to Theorem \ref{4.2} for the explicit construction.

Conversely, if $\mathrm{PGL}_3(K)$ contains a subgroup of one of these three types, then it contains a subgroup of the form $C_3\times C_3$. We are done after applying Lemma \ref{5.2}.
\end{proof}

\begin{proposition}\label{5.4}
The group $\mathrm{PGL}_3(K)$ contains a subgroup of type (H) if and only if $\sqrt{5}\in K$.
\end{proposition}
\begin{proof}
If we have $A_5$, then we have an element of order 5, then $\sqrt{5}\in K$, by Proposition \ref{5.1}.

Conversely, if $\sqrt{5}\in K$, it's known that $$A_5\cong\left\langle x,y\mid x^5=y^2=(xy)^3=1\right\rangle,$$ according to \cite[p140, Beispiel 19.9]{Huppert}. Here $y$ plays the role of (12)(45), and $x$ plays the role of (12345).

Take 
\begin{equation}\label{1}
x=\begin{pmatrix} 
-\sigma & -\tau & -1\\
-\tau & -1 &-\sigma\\ 
1 & \sigma& \tau 
\end{pmatrix}, \ y=\begin{pmatrix}  1 &\sigma&\tau \\ \sigma& \tau & 1\\ \tau & 1 & \sigma \end{pmatrix},
\end{equation}
where $\sigma=\frac{1-\sqrt{5}}{2},\ \tau=\frac{1+\sqrt{5}}{2}$. Then $x$ and $y$ satisfy all the relations, hence they generate a group which is a quotient of $A_5$. It should be either $A_5$ or $\{1\}$ since~$A_5$ is a simple group. They cannot generate the trivial group because the order of $x$ is greater than 1.
\end{proof}

\begin{proposition}\label{5.6}
The group $\mathrm{PGL}_3(K)$ contains a subgroup of type (I) if and only if $\sqrt{-7}\in K$.
\end{proposition}
\begin{proof}
Note that $\mathrm{PSL}_2(\mathbb{F}_7)$ can be generated by its order-$2$ elements. Therefore, according to Lemma \ref{5.5} there is a finite subgroup $G\subset \mathrm{SL}_3(K)$ with $|G|$=168 or 504 which projects to $\mathrm{PSL}_2(\mathbb{F}_7)\subset \mathrm{PGL}_3(K)$. The only finite primitive subgroup of $\mathrm{SL}_3(K)$ (if there is any) with order 504 is $C_3\times \mathrm{PSL}_2(\mathbb{F}_7)$, so we always have a copy of $\mathrm{PSL}_2(\mathbb{F}_7)$ in $\mathrm{SL}_3(K)$. That is, we could give a three dimensional representation of $\mathrm{PSL}_2(\mathbb{F}_7)$ in $K$, which is clearly not a direct sum of three trivial one-dimensional representations. By checking the character table of $\mathrm{PSL}_2(\mathbb{F}_7)$ (see \cite[Section 1.1]{Elkies}), the square root of $-7$ will appear in $K$.

Conversely, formulas (1.6) and (1.7) in \cite{Elkies} give the explicit construction of~$\mathrm{PSL}_2(\mathbb{F}_7)$ when $\sqrt{-7}$ lies in $K$.
\end{proof}

\begin{proposition}\label{5.7}
The group $\mathrm{PGL}_3(K)$ contains a subgroup of type (L) if and only if $\sqrt{5}\in K$ and $\omega\in K$.
\end{proposition}
\begin{proof}
If $\sqrt{5}$ and $\omega$ lie in $K$, consider the group presentation given by ATLAS of Finite Group Representations \cite{ATLAS}:
$$A_6\cong <a,b\mid a^2=b^4=(ab)^5=1>.$$ 
Here in our case $a=(12)(45), b=(1243)(56)$.

We may pick $$z=\begin{pmatrix} 1&0&0 \\ 0&0&\omega \\ 0&\omega^2&0 \end{pmatrix},$$ 
and set $a=y, b=y(zx)^4$, where $x$ and $y$ are the matrices defined in equation~(\ref{1}). Then $a$ and $b$ satisfy all the relations, hence they generate a group which is a quotient of $A_6$. It should be either $A_6$ or $\{1\}$ since $A_6$ is a simple group. They cannot generate the trivial group because the order of $a$ is greater than 1.

Conversely, if we have $A_6$ in $\mathrm{PGL}_3(K)$, we will have $\sqrt{5}\in K$ as $A_6$ contains~$A_5$. Moreover, from Lemma \ref{5.5} we know there is a finite subgroup~\mbox{$G\subset \mathrm{SL}_3(K)$} with $|G|$=360 or 1080 whose projection to $\mathrm{PGL}_3(K)$ is isomorphic to $A_6$. However, $\mathrm{SL}_3(\bar{K})$ does not have a finite subgroup isomorphic to $A_6$, since the minimal dimension for a nontrivial representation of $A_6$ is~5, according to \cite{GroupNamesA6}. One can also derive this from Theorem \ref{4.2} alternatively. Therefore~$|G|=1080$ and~\mbox{$\omega\in K$}.
\end{proof}

\section{Jordan constant of $\mathrm{PGL}_3(K)$}

In this section, we prove our main result Theorem \ref{thm1.4}.

\begin{lemma}\label{6.1}
Let $G$ be a finite subgroup of $\mathrm{PGL}_3(K)$.

1. If G is a group of type (C), then $J(G)\leq 3$.

2. If G is a group of type (D), then $J(G)\leq 6$.
\end{lemma}
\begin{proof}
According to the construction given in Theorem \ref{4.2}, we have a surjective group homomorphism from a group of type (D) to $S_3$, with kernel of type (A), which is abelian, normal, and has index $|S_3|=6$. Similar arguments apply for groups of type (C).
\end{proof}

\begin{lemma}\label{6.2}
%1. If G is a group of type (E), then $J(E)=4$

%2. If G is a group of type (F), then $J(F)=8$

If $G$ is a group of type (G), then $J(G)=24$.
\end{lemma}
\begin{proof}
According to \cite{GroupNames}, the only non-trivial abelian normal subgroup of $G$ is~\mbox{$C_3\times C_3$}, whose index is 24.
\end{proof}

Now we are ready to prove Theorem \ref{thm1.4}:
\begin{proof}
Note that $\mathrm{PGL}_3(K)$ always contains $S_4$: the action of $S_4$ on $K^4$ via permutation of basis $\{e_i\}_{i=1,\ldots,4}$ gives rise to a four-dimensional representation of $S_4$. It has a three-dimensional invariant subspace spanned by $$\{e_1-e_2, e_2-e_3, e_3-e_4\}.$$ So $J(\mathrm{PGL}_3(K))$ is at least 6. 

As a result, we don't need to care about groups of type (A), (C) and (D) as group of type (A) is abelian, and Jordan constant of group of type (C) or (D) cannot be more than 6, which is shown in Lemma \ref{6.1}. For group of type (B), we have a complete result in Theorem \ref{thm1.3}. Also, subgroups of type (E) and (F) always appear together with group of type (G), and the group of type (G) gives the largest Jordan constant among them, so we can omit groups of type (E) and~(F).

The Jordan constants of all possible finite subgroups of $\mathrm{PGL}_3(K)$, excepting those excluded in the previous paragraph, are 24, 60, 168 and 360, coming from a finite group $G$ of type (G), (H), (I) and (L), respectively. Combining with Theorem \ref{thm1.3}, all the possible Jordan constants of $\mathrm{PGL}_3(K)$ are:~\mbox{6, 12, 24, 60, 168} and 360.
% \begin{itemize}
%     \item 12, if $-1$ is the sum of two squares in $K$, and none of $\sqrt{5}, \sqrt{2}$ and $\sqrt{-2}$ is in $K$;
%     \item 24, if (a) either $-1$ is the sum of two squares in $K$, $\sqrt{5}\not\in K$ and at least one of $\sqrt{2}\in K$ and $\sqrt{-2}$ lies in $K$, 
    
% \hspace{22pt} (b) or there exists subgroup $G$ of type (G);
%     \item 60, if $-1$ is the sum of two squares in $K$ and $\sqrt{5}\in K$ or there exists subgroup $G$ of type (H);
%     \item 168, if there exists subgroup $G$ of type (I); and 
%     \item 360, if there exists subgroup $G$ is of type (L).  
% \end{itemize}

We start from the largest possible Jordan constant 360, given by $A_6$, and the criterion for existence of $A_6$ is given in Proposition \ref{5.7}. If 360 cannot be achieved, then we search for the group contributing Jordan constant 168, which is of type (I). Its existence is examined by Proposition \ref{5.6}. If~\mbox{$J(\mathrm{PGL}_3(K))<168$}, we next search for finite subgroup $G$ such that $J(G)=60$, which is given by Theorem \ref{thm1.3} and Proposition \ref{5.4}. However, the condition for Theorem \ref{thm1.3} is stronger, so we only require Proposition \ref{5.4}. Next, if~\mbox{$J(\mathrm{PGL}_3(K))<60$}, we look for the conditions such that $J(\mathrm{PGL}_3(K))=24$, which is given by Theorem \ref{thm1.3} and Proposition \ref{5.3}. If $J(\mathrm{PGL}_3(K))<24$, we then look for the conditions such that $J(\mathrm{PGL}_3(K))=12$. This Jordan constant can only be achieved by groups of type (B), and the criterion is stated in~\mbox{Theorem}~\ref{thm1.3}.
% It is given by Theorem \ref{thm1.3}. 
If none of above conditions is satisfied, we see that~\mbox{$J(\mathrm{PGL}_3(K))=6$}, according to the discussion in the first paragraph of the proof.
% The proof of claim (vi) tells us that the minimal possible $J(\mathrm{PGL}_3(K))$ is $6$, and groups of type (A), (C) and (D) cannot contribute anymore. However, if there is a group of type (B) isomorphic to $A_5$, then there will also exist a group of type (H). Moreover, note that groups of type (E), (F) and (G) always appear simultaneously in $\mathrm{PGL}_3(K)$, thus only the group with largest Jordan constant needs to be considered. Groups of type (E) and (F) have obvious abelian normal subgroup $C_3\times C_3$, so their Jordan constants will definitely be less than that of group of type (G). Eventually, the remaining subgroups are groups of type (G), (H), (I) and (L), and the Jordan constants of them are 24, 60, 168 and 360, respectively.

% Therefore, we obtain claim (i) by applying Proposition 4.10; (ii) by Proposition 4.9; (iii) by Proposition 4.6 and (iv) is indicated by Proposition 4.5.
\end{proof}

\bibliographystyle{plain} % We choose the "plain" reference style
\bibliography{main} % Entries are in the refs.bib file

%\printbibliography

\end{document}